\documentclass{article}

\usepackage{graphicx}
\usepackage{indentfirst}
\usepackage{amsmath,amsfonts,amsthm,amssymb,amsrefs}
\usepackage{mathrsfs,stmaryrd}
\usepackage{xcolor}
\usepackage[all]{xy}
\usepackage{hyperref}

\usepackage{relsize}

\def\co{\colon\thinspace}
\DeclareMathAlphabet{\mathsfsl}{OT1}{cmss}{m}{sl}

\newtheorem{thm}{Theorem}[section]
\newtheorem{lem}[thm]{Lemma}

\newtheorem{prop}[thm]{Proposition}

\theoremstyle{definition}
\newtheorem{defn}[thm]{Definition}

\newtheorem{rem}[thm]{Remark}

\begin{document}

\title{A note on knot Floer homology and fixed points of monodromy}

\author{{\Large Yi NI}\\{\normalsize Department of Mathematics, Caltech, MC 253-37}\\
{\normalsize 1200 E California Blvd, Pasadena, CA
91125}\\{\small\it Emai\/l\/:\quad\rm yini@caltech.edu}}

\date{}
\maketitle

\begin{abstract}
Using an argument of Baldwin--Hu--Sivek, we prove that if $K$ is a hyperbolic fibered knot with fiber $F$ in a closed, oriented $3$--manifold $Y$, and $\widehat{HFK}(Y,K,[F], g(F)-1)$ has rank $1$, then the monodromy of $K$ is freely isotopic to a pseudo-Anosov map with no fixed points. In particular, this shows that the monodromy of a hyperbolic L-space knot is freely isotopic to a map with no fixed points.
\end{abstract}

\section{Introduction}

Knot Floer homology, defined by Ozsv\'ath--Szab\'o \cite{OSzKnot} and Rasmussen \cite{RasThesis}, is a powerful knot invariant. It contains a lot of information about the topology of the knot. For example, it detects the Seifert genus of a knot $K$ \cite{OSzGenus}, and it determines whether $K$ is fibered \cite{Gh,NiFibred}. In fact, such information is contained in $\widehat{HFK}(K,g(K))$, the first nontrivial term of the knot Floer homology with respect to the Alexander grading, which is often referred as the ``topmost term'' in knot Floer homology.

In recent years, it became clear that $\widehat{HFK}(K,g(K)-1)$, which is often called ``the second term'' or ``the next-to-top term'' in knot Floer homology, also contains interesting information about the knot. For example, Lipshitz--Ozsv\'ath--Thurston \cite{LOTFaithful} showed that the ``second term'' of the bordered Floer homology of a surface mapping class completely determines this mapping class. Baldwin and Vela-Vick \cite{BVV} proved that the second term of the knot Floer homology of a fibered knot is always nontrivial, and Ni \cite{NiSecondTerm} proved the same result for knots whose topmost term is supported in a single $\mathbb Z/2\mathbb Z$--grading.

The main theorem in this paper is also of this type.

\begin{thm}\label{thm:FPfree}
Let $Y$ be a closed, oriented $3$--manifold, and $K\subset Y$ be a hyperbolic fibered knot with
fiber $F$ and monodromy $\varphi$. If \[\mathrm{rank}\widehat{HFK}(Y,K,[F],g(F)-1)=1,\] then $\varphi$ is freely isotopic to a pseudo-Anosov map with no fixed points. 
\end{thm}

%\begin{rem}
%It is natural to ask what happens if $K$ is not hyperbolic. In this case, using the full strength of Cotton-Clay's work \cite{CC}, we can also get information about the monodromy of $K$ using the same method in our paper. However, since the most interesting application of  Theorem~\ref{thm:FPfree} would be on L-space surgeries, and surgeries on satellite knots are usually easier to understand, we decide not to address this case for the simplicity of our argument.
%\end{rem}

Recall that a knot $K\subset S^3$ is an {\it L-space knot} if a positive surgery on $K$ is an L-space.
The most interesting scenario to apply Theorem~\ref{thm:FPfree} is when $K\subset S^3$ is a hyperbolic L-space knot. In this case it is well-known that $\mathrm{rank}\widehat{HFK}(S^3,K,g(K)-1)=1$ \cite{OSzLens}. We hope Theorem~\ref{thm:FPfree} will shed more light on the understanding of L-space knots.

The proof of Theorem~\ref{thm:FPfree} uses a strategy due to Baldwin--Hu--Sivek \cite{BHS}, who proved Theorem~\ref{thm:FPfree} when $K$ has the same $\widehat{HFK}$ as the cinquefoil $T_{5,2}$.
In \cite{BHS}, the authors made use of the zero surgery formula on alternating knots in Heegaard Floer homology due to Ozsv\'ath--Szab\'o \cite{OSzAlternating}. We basically replace this result with a more general zero surgery formula. For simplicity, we only state here a formula for knots in $S^3$. If we are just interested in proving 
Theorem~\ref{thm:FPfree} for knots in $S^3$, this case will suffice. The zero surgery formula used in the proof of Theorem~\ref{thm:FPfree} is Proposition~\ref{prop:ZeroSurg}.

Given a null-homologous knot $K\subset Y$, let $Y_{p/q}(K)$ be the manifold obtained by $\frac pq$--surgery on $K$.

\begin{prop}\label{prop:ZeroSurgS3}
Let $K\subset S^3$ be a hyperbolic fibered knot with
genus $g\ge3$. Suppose that the monodromy of $K$ is neither left-veering nor right-veering, then
\[
\mathrm{rank}\:HF^+(S^3_0(K),g-2)=\mathrm{rank}\widehat{HFK}(S^3,K,g-1)-2.
\]
\end{prop}

%Proposition~\ref{prop:ZeroSurg} implies that $\mathrm{rank}\widehat{HFK}(Y,K,g(K)-1)\ge2$, which was already proved in \cite{NiCharSlope} using a similar argument.

During the course of this work, the author learned from John Baldwin that Theorem~\ref{thm:FPfree} is a special case of a theorem announced by Ghiggini and Spano \cite{GS}, which states that $\widehat{HFK}(Y,K,g(K)-1)$ is isomorphic to $HF^{\#}(\varphi)$, a version of the symplectic Floer homology of the monodromy $\varphi$.

This paper is organized as follows. In Section~\ref{sect:Pre}, we recall some basic results about mapping classes. In Section~\ref{sect:ZeroSurg}, we use the zero surgery formula in Heegaard Floer homology to prove Proposition~\ref{prop:ZeroSurg} and Proposition~\ref{prop:ZeroSurgS3}. In Section~\ref{sect:Main}, we prove Theorem~\ref{thm:FPfree}.

\vspace{5pt}\noindent{\bf Acknowledgements.}\quad  The author was
partially supported by NSF grant number DMS-1811900. The author wishes to thank John Baldwin for many helpful discussions and comments on a draft of this paper.

%%%%%
%%%%%
%%%%%
%%%%%
%%%%%

\section{Preliminaries on mapping classes}\label{sect:Pre}

In this section, we recall Thurston's classification of mapping classes and the concept of right-veering diffeomorphisms.

%\subsection{Thurston's classification and right-veering diffeomorphisms}%\label{sect:RV}

Given a compact oriented surface $F$ with boundary, Thurston \cite{ThurstonSurface} classified automorphisms of $F$ as follows. Every automorphism $\varphi$ falls into exactly one of $3$ classes:
\begin{itemize}
\item Periodic: $\varphi$ is freely isotopic to a periodic map $\widetilde{\varphi}$ with $\widetilde{\varphi}^n=\mathrm{id}$ for some integer $n\ge1$.

\item Pseudo-Anosov: $\varphi$ is freely isotopic to a pseudo-Anosov map $\widetilde{\varphi}$. That is, there exist two singular measured foliations $(\mathcal F^u,\mu^u),(\mathcal F^s,\mu^s)$ of $F$, which are transverse everywhere except at the singular points, such that
\[
\widetilde{\varphi}(\mathcal F^u,\mu^u)=\lambda(\mathcal F^u,\mu^u),\quad \widetilde{\varphi}(\mathcal F^s,\mu^s)=\lambda^{-1}(\mathcal F^s,\mu^s)
\]
for a fixed real number $\lambda>1$.

\item Reducible: $\varphi$ is freely isotopic to a reducible map $\widetilde{\varphi}$. That is, there exists a collection $\mathcal C$ of mutually disjoint essential simple closed curves, such that $\widetilde{\varphi}(\mathcal C)=\mathcal C$. Moreover, $F\setminus\mathcal C$ can be divided into (possibly disconnected) subsurfaces $F_1,\dots,F_m$, such that the mapping class of $\widetilde{\varphi}|F_i$ is either periodic or pseudo-Anosov. We also require that $\mathcal C$ is the minimum collection of curves with this property, and $\mathcal C\ne\emptyset$.
\end{itemize}

Thurston proved that the mapping torus of $\varphi$ has a complete, finite volume hyperbolic structure in its exterior if and only if the mapping class of $\varphi$ is pseudo-Anosov.

Let $\varphi: F\to F$ be a diffeomorphism such that $\varphi(\partial F)=\mathrm{id}_{\partial F}$.
By Thurston's classification, $\varphi$ is freely isotopic to a standard representative $\widetilde{\varphi}$, whose restriction to $\partial F$ may not be the identity. For each component $C$ of $\partial F$, one can define a {\it fractional Dehn twist coefficient (FDTC)} $c(\varphi)\in\mathbb Q$, which is the rotation number of $\widetilde{\varphi}$ on $C$ compared with $\varphi$. This concept essentially appeared in work of Gabai \cite[Remark~8.7~ii)]{GabaiProblems},
and the term FDTC was coined by Honda--Kazez--Mati\'c \cite{HKMRightVeering}.

Honda--Kazez--Mati\'c \cite{HKMRightVeering} also defined right-veering diffeomorphisms.

\begin{defn}
Let $F$ be a compact surface with boundary, $a,b\subset F$ be two properly embedded arcs with $a(0)=b(0)=x$. We isotope $a,b$ with endpoints fixed, so that $|a\cap b|$ is minimal.
We say $b$ is {\it to the right of $a$}, denoted $a\le b$, if either $b$ is isotopic to $a$ with endpoints fixed, or $(b\cap U)\setminus\{x\}$ lies in the ``right'' component of $U\setminus a$, where $U\subset F$ is a small neighborhood of $x$. See Figure~\ref{fig:ToRight}.
%In this case, we also say that 
%$a$ is {\it to the left of $b$} at $x$. If $b$ is not isotopic to $a$ with endpoints fixed, we say $b$ is {\it strictly to the right of $a$} at $x$, and $a$ is {\it strictly to the left of $b$} at $x$. 
\end{defn}

\begin{figure}[ht]
\begin{picture}(340,66)
\put(100,6){\scalebox{0.5}{\includegraphics*%[0pt,0pt][100pt,100pt]
{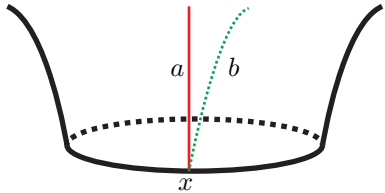}}}

\put(162,43){$a$}

\put(184,43){$b$}

\put(165,0){$x$}

\end{picture}
\caption{\label{fig:ToRight}The arc $b$ is to the right of $a$.}
\end{figure}

\begin{defn}
Let $\varphi\co F\to F$ be a diffeomorphism that restricts to the identity map on $\partial F$. Let $C$ be a component of $\partial F$. Then $\varphi$ is {\it right-veering} with respect to $C$ if for every $x\in C$ and every properly embedded arc $a\subset F$ with $x\in a$, the image $\varphi(a)$ is to the right of $a$ at $x$. Similarly, we can define {\it left-veering} with respect to $C$. If $\varphi$ is right-veering with respect to every component of $\partial F$, we say $\varphi$ is a {\it right-veering} diffeomorphism. Similarly, we can define {\it left-veering} diffeomorphisms.
\end{defn}

If the mapping class of $\varphi$ is pseudo-Anosov, then $\varphi$ is right-veering with respect to $C$ if and only if $c(\varphi)>0$ for $C$ \cite[Proposition~3.1]{HKMRightVeering}. If the mapping class of $\varphi$ is periodic, then $\varphi$ is right-veering with respect to $C$ if $c(\varphi)>0$ for $C$ \cite[Proposition~3.2]{HKMRightVeering}.

%\subsection{Symplectic Floer homology of a mapping class}

%%%%%
%%%%%
%%%%%
%%%%%
%%%%%

\section{Knot Floer homology and the zero surgery formula}\label{sect:ZeroSurg}

In this section, we will use arguments in \cite{NiCharSlope,NiSecondTerm} to prove Propositions~\ref{prop:ZeroSurgS3} and~\ref{prop:ZeroSurg}. We assume the readers are reasonably familiar with Heegaard Floer homology.

\begin{prop}\label{prop:ZeroSurg}
Let $Y$ be a closed, oriented $3$--manifold, and $K\subset Y$ be a hyperbolic fibered knot with
fiber $F$.
 Let $\mathfrak s\in\mathrm{Spin}^c(Y)$ be the underlying Spin$^c$ structure for the open book decomposition of $Y$ with binding $K$ and page $F$. Let $\widehat F\subset Y_0(K)$ be the closed surface obtained by capping off $\partial F$ with a disk, and let $\mathfrak t_k\in\mathrm{Spin}^c(Y_0(K))$ be the Spin$^c$ structure satisfying
\[
\mathfrak t_k|(Y\setminus K)=\mathfrak s|(Y\setminus K),\quad \langle c_1(\mathfrak t_k),\widehat{F}\rangle=2k,\quad k\in\mathbb Z.
\]
Suppose that either $\mathfrak s$ is torsion, or $HF^+(Y,\mathfrak s)=0$.
Suppose also that the monodromy of $K$ is neither left-veering nor right-veering. If \[\mathrm{rank}\widehat{HFK}(Y,K,[F],g(F)-1)=2\] and
$g(F)\ge3$, then
\[
\mathrm{rank}\:HF^+(Y_0(K),\mathfrak t_{g-2})=0.
\]
\end{prop}

Let $Y$ be a closed, oriented $3$--manifold, $K\subset Y$ be a fibered knot with a fiber $F$ of genus $g$, and let $\mathfrak s\in\mathrm{Spin}^c(Y)$ be the Spin$^c$ structure of the open book decomposition with binding $K$ and page $F$. By \cite{OSzContact}, there exists a Heegaard diagram for $(Y,K)$, such that the topmost knot Floer chain complex $\widehat{CFK}(Y,K,[F],g)$ has a single generator. Let 
\[
C=CFK^{\infty}(Y,K,\mathfrak s,[F])
\]
be the knot Floer chain complex. Let $\partial$ be the differential on $C$, and let $\partial_0$ be the summand of $\partial$ which preserves the $(i,j)$-grading.

We will consider
the chain complex
$C\{i<0, j\ge g-2\}$, which has the form
\begin{equation}\label{eq:MCg-2}
\xymatrixcolsep{5pc}
\begin{xymatrix}{  
& C(-1,g-1)\ar[d]^{\partial_z}\ar[dl]_{\partial_{zw}}\\
C(-2,g-2)&C(-1,g-2)\ar[l]_{\partial_w},}
\end{xymatrix}
\end{equation}
where
\begin{equation}\label{eq:TopZ}
C(-i,g-i)\cong \widehat{CFK}(Y,K,\mathfrak s,[F],g)\cong\mathbb Z, \text{ for all }i\in\mathbb Z, 
\end{equation}
and
\[ C(-1,g-2)\cong\widehat{CFK}(Y,K,\mathfrak s,[F],g-1).
\]

By (\ref{eq:TopZ}), $\partial_0=0$ on $C(-1,g-1)$ and $C(-2,g-2)$. Since $\partial^2=0$, we see that 
\begin{equation}\label{eq:dwdz}
\partial_w\partial_z=0
\end{equation} 
and $\partial_{zw}^2=0$, where we use (\ref{eq:TopZ}) to identify every $C(-i,g-i)$ with $\mathbb Z$. So we must have
\[
\partial_{zw}=0.
\]

\begin{lem}\label{lem:ZeroMC}
If the monodromy of $K$ is neither left-veering nor right-veering, then
the rank of the homology of the mapping cone (\ref{eq:MCg-2}) is \[\mathrm{rank}\widehat{HFK}(Y,K,\mathfrak s,[F],g-1)-2.\]
\end{lem}
\begin{proof}
As we have showed above, the mapping cone (\ref{eq:MCg-2}) becomes
the chain complex
\begin{equation}\label{eq:NewMCg-2}
\xymatrixcolsep{5pc}
\begin{xymatrix}{  
& \mathbb Z\ar[d]^{\partial_z}\\
\mathbb Z&C(-1,g-2)\ar[l]_{\partial_w}.}
\end{xymatrix}
\end{equation}

Since the monodromy of $K$ is neither left-veering nor right-veering, by the argument in the proof of \cite[Theorem~1.1]{BVV}, the induced map \[(\partial_z)_*:\mathbb Z\to H_*(C(-1,g-2))\] is injective, and the induced map
\[(\partial_w)_*:H_*(C(-1,g-2))\to \mathbb Z\] is surjective. Using (\ref{eq:dwdz}),
we see that the rank of the homology of (\ref{eq:NewMCg-2}) is
\[
\mathrm{rank}\:H_*(C(-1,g-2))-2=\mathrm{rank}\:\widehat{HFK}(Y,K,\mathfrak s,[F],g-1)-2.\qedhere
\]
\end{proof}

\begin{proof}[Proof of Proposition~\ref{prop:ZeroSurgS3}]
By a standard argument, (see, for example, \cite[Corollary~4.5]{OSzKnot},) $HF^+(S^3_0(K),g-2)$ is isomorphic to the homology of (\ref{eq:MCg-2}), so our conclusion follows from 
Lemma~\ref{lem:ZeroMC}.
\end{proof}

As in \cite{OSzIntSurg}, for any $k\in\mathbb Z$, let
$$A^+_k=C\{i\ge0 \text{ or }j\ge k\},\quad k\in\mathbb Z$$
and $B^+=C\{i\ge0\}\cong CF^+(Y,\mathfrak s)$. There are chain maps
$$v^+_k,h^+_k\co A^+_k\to B^+.$$
Here $v^+_k$ is the vertical projection, and $h^+_k$ is essentially the horizontal projection.

\begin{proof}[Proof of Proposition~\ref{prop:ZeroSurg}]
It is well known that $HF^+(Y_0(K),\mathfrak t_{g-2})$ is isomorphic to the homology of $MC(v^+_{g-2}+h^+_{g-2})$, the mapping cone of 
\[
v^+_{g-2}+h^+_{g-2}\co A^+_k\to B^+.
\]
In fact, when $\mathfrak s$ is torsion, this result follows from the same argument as in \cite[Subsection~4.8]{OSzIntSurg}; when $\mathfrak s$ is non-torsion, the formula is proved in \cite[Theorem~3.1]{NiPropertyG}.

By the exact triangle
\begin{equation}\label{eq:v(g-2)}
\xymatrix{
H_*(A^+_{g-2})\ar[r]^-{(v^+_{g-2})_*}&H_*(B^+)\ar[ld]\\ % &HF^+(Y_n(K),[\mathfrak s_{d-1}])\ar[r]^-{F_{W'_n(K)}}&HF^+(Y,\mathfrak s)\ar[ld]\\
H_*(C\{i<0, j\ge g-2\})\ar[u]&  %&HF^+(Y_0(K),\mathfrak t_{d-1})\ar[u]&
},
\end{equation}
$C\{i<0, j\ge g-2\}$ is quasi-isomorphic to $MC(v^+_{g-2})$, the mapping cone of $v^+_{g-2}$.

By Lemma~\ref{lem:ZeroMC}, $H_*(MC(v^+_{g-2});\mathbb Q)=0$, so $(v^+_{g-2})_*$ is an isomorphism over $\mathbb Q$. 

If $\mathfrak s$ is torsion, there is an absolute $\mathbb Q$--grading on $A^+_{k}$ and $B^+$. Since $g-2>0$, the grading shift of $h^+_{g-2}$ is strictly less than the grading shift of $v^+_{g-2}$. Hence  $(v^+_{g-2})_*+(h^+_{g-2})_*$ is also an isomorphism over $\mathbb Q$, which implies that \begin{equation}\label{eq:Y0=0}
HF^+(Y_0(K),\mathfrak t_{g-2};\mathbb Q)=0.\end{equation}

If $HF^+(Y,\mathfrak s)=0$, $H_*(B^+)=0$, so $H_*(A^+_{g-2};\mathbb Q)=0$ since  $(v^+_{g-2})_*$ is an isomorphism over $\mathbb Q$. We again have (\ref{eq:Y0=0}).
\end{proof}

%%%%%
%%%%%
%%%%%
%%%%%
%%%%%

\section{Proof of the main theorem}\label{sect:Main}

\begin{lem}\label{lem:HypS1S2}
There exists a hyperbolic fibered knot $L\subset Z=S^1\times S^2$ with fiber $G$, such that the monodromy of the fibration is right-veering, and the Spin$^c$ structure of the open book decomposition with binding $L$ and page $G$ is non-torsion.
\end{lem}
\begin{proof}
By \cite{Eliashberg}, there exists a non-torsion contact structure $\xi$ on $Z$. Let $(S,h)$ be an open book decomposition supporting $\xi$. By \cite{CH}, we can stabilize $(S,h)$ many times to get a new open book $(G,\psi)$ with connected binding $L$, such that the  monodromy $\psi$ is pseudo-Anosov and right-veering.
\end{proof}

Let $L$ be the knot as in Lemma~\ref{lem:HypS1S2}, and let $L'$ be the $(2n+1,2)$--cable of $L$ for a sufficiently large integer $n$. Let $E$ be the fiber of the new fibration of $Z$ with $\partial E=L'$, then 
\[
E=T\cup G_1\cup G_2,
\]
where $T$ is a genus $n$ surface with $3$ boundary components, and $G_1,G_2$ are two copies of $G$. Let $\rho$ be the monodromy of $L'$. Then $\rho|T$ is isotopic to a periodic map of period $4n+2$, and $\rho$ swaps $G_1,G_2$.

\begin{lem}\label{lem:4n+2}
The FDTC of $\rho$ on $L'$ is $\frac1{4n+2}$.
\end{lem}
\begin{proof}
The complement of $L'$ is the union of a cable space and $Z\setminus L$. The slope on $L'$ of the Seifert fiber of the cable space is $4n+2$.
Our conclusion follows from the definition of FDTC.
\end{proof}

\begin{proof}[Proof of Theorem~\ref{thm:FPfree}]
The proof of Theorem~\ref{thm:FPfree} uses a similar argument as \cite[Theorem~3.5]{BHS}. The new input here is to replace  \cite[Equation~(3.3)]{BHS} with Proposition~\ref{prop:ZeroSurg}.

Since $\mathrm{rank}\widehat{HFK}(Y,K,[F],g(F)-1)=1$, it follows from \cite[Theorem~A.1]{NiCharSlope} that $\varphi$ is either right-veering or left-veering. Without loss of generality, we assume $\varphi$ is right-veering.

Let $L'$ be as in Lemma~\ref{lem:4n+2}. By \cite{Hedden}, we have
\[
\widehat{HFK}(Z,L',[E],g(E)-1)\cong\mathbb Z.
\]

Consider the connected sum $K\#\overline{L'}$, which is a knot in $Y\#Z$. Let $g'$ be the genus of the Seifert surface $F\natural \overline E$.
 By the K\"unneth formula, 
\begin{equation}\label{eq:2r}
\mathrm{rank}\widehat{HFK}(Y\#Z,K\#\overline{L'},[F\natural \overline E], g'-1)=2.
\end{equation}
The monodromy of $K\#\overline{L'}$ is a map $\sigma$ on $F\natural \overline E$. By Lemma~\ref{lem:4n+2}, $\sigma|\overline{E}$ is left-veering, so $\sigma$ is neither left-veering nor right-veering. 

Let $\mathfrak s\in \mathrm{Spin}^c(Y\#Z)$ be the Spin$^c$ structure of the open book decomposition with binding $K\#\overline{L'}$ and page $F\natural \overline E$.
Since the restriction of $\mathfrak s$ to $Z\setminus B^3$ is non-torsion, $HF^+(Y\#Z,\mathfrak s)=0$ by the adjunction inequality \cite{OSzAnn2}.

Now Proposition~\ref{prop:ZeroSurg} and (\ref{eq:2r}) imply that
\begin{equation}\label{eq:2r-2}
\mathrm{rank}\:HF^+((Y\#Z)_0(K\#\overline{L'}),g'-2)=0.
\end{equation}

The manifold $(Y\#Z)_0(K\#\overline{L'})$ is a surface bundle over $S^1$. Its fiber $P$ is a closed surface which is the union of $F$ and $\overline E$.
Let $\widehat{\sigma}$ be the monodromy, then $\widehat{\sigma}|F=\varphi$.

The rest of our argument is similar to \cite[Theorem~3.5]{BHS}. Using work of Lee--Taubes \cite{LeeTaubes}, Kutluhan--Lee--Taubes \cite{KLT}, Kronheimer--Mrowka \cite{KMBook}, one sees that $HF^+((Y\#Z)_0(K\#\overline{L'}),g'-2)$ is isomorphic to the symplectic Floer homology $HF^{symp}_*(P,\widehat{\sigma})$ of $(P,\widehat{\sigma})$. By (\ref{eq:2r-2}), $HF^{symp}_*(P,\widehat{\sigma};\mathbb Q)=0$. 

Recall that we assume $\varphi$ is right-veering. In the terminology of \cite{CC}, $\varphi$ has no Type IIId fixed points.
 By \cite[Theorem~4.16]{CC}, $HF^{symp}_*(P,\widehat{\sigma};\mathbb Q)$ contains a direct summand which is freely generated by a superset of the fixed points of the pseudo-Anosov representative $\widetilde{\varphi}$ of $\varphi$. So $\widetilde{\varphi}$ has no fixed points. 
\end{proof}

\begin{rem}
In \cite{BHS}, the authors computed $S^3_0(K\#\overline K, g(K\#\overline K)-2)$ to get the conclusion. This approach will also work in the general case if the underlying Spin$^c$ structure of the open book decomposition corresponding to $K$ is torsion.
\end{rem}

%%%%%
%%%%%
%%%%%
%%%%%
%%%%%

\end{document}